\documentclass{amsart}

\usepackage{amssymb, amsmath, enumerate,multicol}
\usepackage[all]{xy}

\newtheorem{thm}{Theorem}[section]
\newtheorem{lem}[thm]{Lemma}
\newtheorem{prop}[thm]{Proposition}
\newtheorem{cor}[thm]{Corollary}

\theoremstyle{definition}
\newtheorem{Def}[thm]{Definition}
\newtheorem{rmk}[thm]{Remark}
\newtheorem{Ex}[thm]{Example}

\numberwithin{equation}{thm}

\newcommand{\Ext}{{\rm Ext}_A}

\newcommand{\Hom}{{\rm Hom}_A}

\newcommand{\ds}\displaystyle
\newcommand{\ed}{{\rm{ext.deg}}}
\newcommand{\fed}{{\rm{fed}}}
\newcommand{\fpd}{{\rm{fpd}}}

\newcommand{\pd}{{\rm pd}}
\newcommand{\id}{{\rm id}}
\newcommand{\im}{{\rm im}}
\newcommand{\Amod}{A{\rm \mbox{-}mod}}
\newcommand{\modA}{A^\textsf{o} {\rm \mbox{-}mod}}

\newcommand{\CM}{{\rm CM}(A)}

\newcommand{\x}{{\underline{a}}}

\newcommand{\comment}[1]{}

\begin{document}

\title{Finitistic extension degree}
\author{Kosmas Diveris}
\address{St.\ Olaf College, Northfield, MN, USA}
\email{diveris@stolaf.edu}
\date{September 10, 2012}
\maketitle 

\begin{abstract} We introduce the finitistic extension degree of a ring and investigate rings for which it is finite.  The Auslander-Reiten Conjecture is proved for rings of finite finitistic extension degree and these rings are also shown to have finite finitistic dimension.  We apply these results to better understand a generalized version of the Auslander-Reiten Condition for Gorenstein rings.  We also record how the finitistic extension degree behaves with respect to many change of ring procedures that arise frequently in the commutative setting.     \end{abstract}

\begin{section}{Introduction}  

A cohomological condition for rings, now known as Auslander's Condition or (AC), arose out of unpublished work of M.\ Auslander related to the finitistic dimension conjecture (cf.\ the introduction to Chapter V of \cite{Aus1}).  While it is known that not all rings satisfy this condition \cite{aus}, several classes of rings which do have been identified and are these rings are known to satisfy longstanding homological conjectures.  

In \cite{AC} L.\ Chirstensen and H.\ Holm undertook a thorough investigation Auslander's Condition.  They showed that many of the properties that rings satisfying (AC) are known to have in common actually follow from Auslander's Condition.  Inspired by the results and remaining questions from \cite{AC}, in the present paper we introduce the finitistic extension degree of a ring $A$, which we denote by $\fed(A)$.  The finitness of this new invariant is a weaker condition than Auslander's Condition.  We are able to recover several results known for rings satisfying (AC), and in some cases we also obtain their converses, under the weaker assumption that the ring has finite finitistic extension degree.

The outline of this article is as follows. In Section 2 we define the finitistic extension degree of a ring and examine the consequences that follow from its finiteness.  In particular, we show that such rings satisfy the Finitistic Dimension Conjecture, the Auslander-Reiten Conjecture and the Gorenstein Symmetry Question. In Section 3 we specialize to rings of finite injective dimension and continue the investigation from the previous section.  In this setting we are able to obtain converses to results from Section 2 and give the precise value of $\fed(A)$ in several important cases.  In the Section 4  we show that finitistic extension degree behaves well with respect to some standard change of rings procedures in the commutative setting.

\end{section}
 
\begin{section}{Finitistic extension degree} 

Throughout this paper we consider a left Noetherian ring $A$ and we denote by $\Amod$ the category of finitely generated left $A$-modules.  We begin by recalling Auslander's Condition concerning the vanishing of cohomology referred to in the introduction. 

\begin{Def} \label{def:AC} For an $A$-module $M$,  the {\it Auslander bound for $M$} is defined to be 
\begin{align*} b_M  = \sup_i \{   \Ext^i(M,N) \neq 0  \ | \  N \in \Amod \ {\rm satisfies} \ \Ext^{i}(M,N) =0 \ {\rm  for\  all} \ i \gg 0 \}
\end{align*}   One says that $A$ satisfies {\it Auslander's Condition}, or (AC), if  $b_M$ is finite for each $A$-module $M$. If, in addition, there is some integer $b$, with the property that $b_M \leq b$ for each $A$-module $M$, one says that $A$ satisfies the {\it Uniform Auslander Condition}, or (UAC).  In this case $b$ is called a {\it uniform Auslander bound} for $A$.  \end{Def}

Examples of rings satisfying Auslander's Condition along with a discussion concerning its relationship with the Uniform Auslander Condition are provided in Appendix A of \cite{AC}.  We simply remark that there are no known examples Artin algebras or commutative local rings of finite Krull dimension which satisfy (AC) but not (UAC).  In fact, any ring of finite injective dimension satisfying (AC) must also satisfy (UAC) \cite{HJ}, \cite{Mori}.

Rather than considering all pairs of modules with eventually vanishing extensions, to define the finitistic extension degree we specialize to those modules having eventually vanishing self-extensions.

\begin{Def}\label{def:fed}  The {\it self-extension degree} of an $A$-module $M$ is defined to be \begin{align*} \ed(M) = \sup\{i \ | \ \Ext^i(M,M) \neq 0 \} \end{align*}  and the {\it finitistic extension degree} of the ring $A$ is defined to be \begin{align*} \fed(A) = \sup\{ \ed(M) \ | \ \ed(M) \ {\rm is \ finite} \} \end{align*} 
\end{Def}

\begin{rmk}\label{rmk:AC.FED}   One sees immediately from these definitions that whenever $A$ satisfies the Uniform Auslander Condition, then $\fed(A)$ is finite. However, these conditions are not equivalent.  In \cite{aus} Jorgensen and \c{S}ega provide an example of a commutative self-injective Artin algera that does not satisfy Auslander's Condition.  This ring does satisfy the hypotheses of Theorem 3.3 in \cite{self-test} and a slight modification of the proof of {\it loc.\ cit.} shows that the finitistic extension degree of this ring is zero (see Theorem 4.21 of \cite{my.thesis} for a complete justification of this claim).   \end{rmk}

We now proceed with the following observation that relates $\ed(M)$ to the projective dimension of $M$, which we denote by $\pd(M)$.
 
\begin{lem} \label{lem:inequality} Let $A$ be a Noetherian ring.  Then the following inequalities hold for each $A$-module $M$:  \begin{align*} \ed(M) \leq \ed(M \oplus A) \leq \pd(M). \end{align*}  In addition, we have:
\newcounter{count}
\begin{list}{(\arabic{count})}{\usecounter{count} \leftmargin=.5em}
\item If $M$ has finite projective dimension, then equality holds on the right. 
\item If $A$ is a local ring and $M$ has finite projective dimension, then equality holds everywhere.
\end{list}
\end{lem}

\begin{proof}  The inequalities are immediate consequences of the definition of extension degree.  To show (1), we assume that $\pd(M)$ is finite. Recall (Ex.\ 9 in  \cite[VI]{CE}) that one then has the equality: \begin{align}  \label{eq:pd.ext} \pd(M) = \sup\{ i \ | \ \Ext^i(M,A) \neq 0 \}. \end{align}  Since $\Ext^i(M,A)$ is a direct summand of $\Ext^i(M \oplus A, M \oplus A)$, we obtain the inequality  $\ed(M \oplus A) \geq \pd(M)$.  This shows that (1) holds.

We  now assume that $A$ is a local ring with maximal ideal $\frak{m}$.  If $M$ has finite projective dimension, we compute its self-extensions from a minimal free resolution: \[ \xymatrix{ 0  \ar[r] &  F_n \ar[r]^-{\partial} & F_{n-1} \ar[r] & \cdots \ar[r]&  F_0 \ar[r] & M \ar[r]  & 0 } \] To show (2), we need to show that $\Ext^n(M,M) \neq 0$.  This extension group is the cokernel of  the map \[ \xymatrix{\Hom(F_{n-1},M) \ar[r]^-{{\partial}^*} & \Hom(F_n,M) } \]  
 The minimality of the above resolution gives $\im( \partial^*(\varphi)) = \im (\varphi \circ \partial) \subseteq \frak{m}M$ for each $\varphi \in \Hom(F_{n-1},M)$.  However, $F_n$ is a non-zero free $A$-module and so there do exist maps in $\Hom(F_{n},M)$ whose image in not contained in $\frak{m}M$.  Therefore $\partial^*$ is not surjective and $\Ext^n(M,M) \neq 0$. \end{proof}

\Def \label{fin.dim} The {\it little finitistic dimension} of the ring $A$ is \begin{align*} {\fpd}(A) = \sup\{ \pd(M) \ | \  \pd(M) \  {\rm is \ finite} \}. \end{align*}

If $A$ is a Noetherian ring, then it has been conjectured that $\fpd(A)$ is finite.  This is known as the Finitistic Dimension Conjecture and was first recorded by H.\ Bass in \cite{Bass}.  If $A$ is commutative and local, then the Auslander-Buchsbaum Theorem  \cite[1.3.3]{BH} shows that $\fpd(A)$ is finite.  For Artin algebras, however, the Finitistic Dimension Conjecture remains open.  We refer to \cite{sheerbrooke} for details concerning the relationship between this conjecture and Auslander's Condition. One easily obtains the following inequality between $\fed(A)$ and $\fpd(A)$ the previous lemma:

\begin{cor}  \label{cor:fpd<fed} For any Noetherian ring $A$, we have $\fpd(A) \leq \fed(A)$. \end{cor}

\begin{proof}  If $\pd(M)$ is finite, then we have $\pd(M) = \ed(M \oplus A) \leq \fed(A)$.   \end{proof}

Dimension shifting is a key ingredient in the proof of the next lemma.  The following remark contains the facts that we will need, see e.g.\ \cite{Weibel} for more details.   We let $\Omega^nM$ denote an $n^{th}$-syzygy of $M$.   While the $n^{th}$-syzygy of a any module depend on the choice of projective resolution, the isomorphism in (\ref{eq:dim.shift}) below is easily seen to be independent of this choice.

\begin{rmk} \label{rmk:dimension.shift} Suppose that $A$ is a Noetherian ring, $M$ is an $A$-module and $d$ is a non-negative integer such that $\Ext^i(M,A) = 0$ for all $i>d$.  If $N$ is any $A$-module and $j,m,n$ are non-negative integers with $d< \min \{j, j-m+n\}$, then \begin{align} \label{eq:dim.shift} \Ext^j(M,N) \cong \Ext^{j-m+n}(\Omega^mM,\Omega^nN) \end{align} Setting $N = M$ and $n = m$, we see $\ed(M) \leq \ed(\Omega^nM)+d$ and these modules have finite extension degree simultaneously. 
\end{rmk}

\begin{lem} \label{lem:dim.shift} Suppose that $A$ is a Noetherian ring, $M$ is an $A$-module and $d$ is a non-negative integer such that $\Ext^i(M,A) = 0$ for all $i>d$.  Then we have $\ed(M \oplus \Omega^nM) \leq \ed(M) + n+d$ for each $n \in \mathbb{N}$, with equality holding if $d=0$ and $\pd(M) > 0$. \end{lem}

\begin{proof}  
We may assume that $\ed(M)$ is finite, otherwise it is evident that equality holds.  We set $m = \ed(M)$ and appeal to Remark \ref{rmk:dimension.shift} to obtain the following vanishing:  \begin{align*}
   \left. \begin{array}{r} \Ext^j(M,M) \\
     \Ext^j(\Omega^nM, \Omega^nM)  \end{array}  \right\} & =  0 \   {\rm for} \ j>\max\{d, m\} \\
   \Ext^j(\Omega^nM,M) & =  0\ {\rm for}  \ j> \max\{m -n,d\} \\
  \Ext^j(M, \Omega^nM) & = 0\  {\rm for}  \ j>\max\{m+n,d\}
\end{align*}  Since the direct sum of these four extension groups is $\Ext^j(M \oplus \Omega^nM, M \oplus \Omega^nM)$, we see $\ed(M \oplus \Omega^nM) \leq m + n+d$, as claimed.

Suppose now that $M$ is not projective and $\Ext^i(M,A) = 0$ for $i>0$.  Given the above inequality, to show $\ed(M \oplus \Omega^nM) = m + n$ it suffices to show that $\Ext^{m+n}(M,\Omega^nM) \neq 0$. When $m > 0$ this follows immediately from (\ref{eq:dim.shift}). When $m = 0$, we have that $\Ext^1(M, \Omega M)$ is nontrivial because $M$ is not a projective module. For larger values of $n$, one may now appeal to (\ref{eq:dim.shift}) in order to obtain $\Ext^n(M, \Omega^n M) \neq 0$. \end{proof} 

In Lemma \ref{lem:inequality} we showed that when $M$ is an $A$-module of finite projective dimension the equality $\pd(M) = \ed(M \oplus A)$ holds.  In the next result we show this equality extends to all $A$-modules when $\fed(A)$ is finite. 

\begin{thm} \label{thm:GARC} If $A$ is a Noetherian ring and $\fed(A)$ is finite, then the equality $\pd(M) = \ed(M \oplus A)$ holds for every $A$-module $M$. \end{thm}

\begin{proof} From Lemma \ref{lem:inequality}, we see that the equality $\pd(M) = \ed(M\oplus A)$ holds when either (i) $\pd(M)$ is finite, or (ii) $\ed(M \oplus A) = \infty$.  We claim that when $\fed(A)$ is finite, then each $A$-module $M$ satisfies either (i) or (ii).  

Suppose, for the sake of contradiction, that there exists an $A$-module $M$ such that $m = \ed(M \oplus A)$ is finite but $\pd(M)$ is infinite. Then we have $\ed(M) \leq m$ and also $\Ext^i(M,A) = 0$ for $i>m$.  We denote by $N$ the syzygy $\Omega^mM$.  Applying the dimension shift from Remark \ref{rmk:dimension.shift} gives that $\ed(N)$ is finite and $\Ext^i(N,A) = 0$ for $i>0$.  Since $M$ has infinite projective dimension, $N$ is not a projective module.   
 
An application of Lemma \ref{lem:dim.shift} now gives the following equality for each non-negative integer $n$, the inequality is clear: \begin{align*} n + \ed(N) = \ed(N \oplus \Omega^nN) \leq \fed(A) \end{align*}  Since $\fed(A)$ is finite, this provides the necessary contradiction.
\end{proof}

An immediate corollary of this result is that a strong version of the Auslander-Reiten Conjecture holds for rings with finite finitistic extension degree.  Before stating this let us recall the definitions.

\begin{Def} \label{def:ARC} We say that $A$ satisfies the {\it Auslander-Reiten Condition} (ARC) if whenever $M$ is an $A$-module such that $\Ext^i(M,M\oplus A) = 0$ for all $i > 0$ one has that $M$ is projective.
\end{Def}

This condition was introduced in \cite{AuR2} where the authors conjectured that all Artin algebras satisfy (ARC).  They then show that this conjecture is equivalent to a generalized version of a conjecture of Nakayama.  This condition also been considered for commutative Noetherian rings.  A natural generalization (ARC) is the following condition: 

\begin{Def} \label{def:GARC} We say that $A$ satisfies the {\it Generalized Auslander-Reiten Condition} (GARC) if whenever $M$ is an $A$-module such that $\Ext^i(M,M\oplus A) = 0$ for all $i >n $ one has that $\pd(M) \leq n$. \end{Def}

It is clear that any ring satisfying (GARC) also satisfies (ARC), indeed it is just the special case when $n=0$. These are not, however, equivalent conditions for Artin algebras (see Remark \ref{rmk:arc.garc} below).

The analogous version of the next result for rings satisfying Auslander's Condition is given in Theorem 2.3 of \cite{AC}.  We note that while they consider a version of (GARC) for complexes rather than just modules, it is equivalent to the version we have stated here by Theorem 3.4 of \cite{D(GARC)}.

\begin{cor} \label{cor:GARC} If $A$ is a Noetherian ring and $\fed(A)$ is finite, then the Generalized Auslander-Reiten Condition holds for $A$. \end{cor}

\begin{proof} Observe that in terms of extension degrees, one may rephrase the condition (GARC) as $\pd(M) \leq \ed(M \oplus A)$ for every $A$-module $M$.  Then the corollary follows immediately from Theorem \ref{thm:GARC}.
\end{proof} 
 
In the sequel we denote by $A^{\textsf{o}}$ the opposite ring of $A$ and identify the category of left $A^{\textsf{o}}$-modules with the that of the right $A$-modules.  We denote the injective dimension of the $A$-module $M$ by $\id_A(M)$.

\begin{Def} The ring $A$ is Gorenstein if both $\id_A(A)$ and $\id_{A^{\textsf{o}}}(A)$ are finite.  \end{Def}

It remains an open question if $A$ Gorenstein when only $\id_A(A)$ is known to be finite, that is whether $\id_A(A) <\infty$ implies  $\id_{A^{\textsf{o}}}(A) < \infty$.  This has been called the Gorenstein Symmetry Question \cite{AC}.

It follows from the next proposition that the Gorenstein Symmetry Question is answered in the affirmative for Artin algebras of finite finitistic extension degree.  Recall, from Chapter II of \cite{ARS}, that when $A$ is an Artin algebra and $J$ is the direct sum of the indecomposable injective $A$-modules then the contravariant functor $D(-) = \Hom(-,J)$ gives a duality  $ \Amod  \rightarrow  \modA$.

\begin{prop} \label{prop:symmetry} If an Artin algebra $A$ satisfies the Generalized Auslander-Reiten Condition, then one has the following: \begin{align*} \id_{A^{\textsf{o}}}(A) =  \ed(D({_{A^{\textsf{o}}}A}) \oplus A) \leq \id_A(A)  \end{align*}  If, in addition, $\id_A(A)$ is finite then equality also holds on the right. \end{prop}

\begin{proof}  As observed in the proof of Corollary \ref{cor:GARC},  we have that when $A$ satisfies (GARC), there is an equality \begin{align*} \pd_A(D({_{A^{\textsf{o}}}A})) = \ed(D({_{A^{\textsf{o}}}A}) \oplus A) \end{align*} The equality  $\id_{A^{\textsf{o}}}(A) = \pd_A(D({_{A^{\textsf{o}}}A}))$ is provided by Lemma 6.9 in \cite{AuR}.  Together, these give the desired equality $\id_{A^{\textsf{o}}}(A) =  \ed(D({_{A^{\textsf{o}}}A}) \oplus A)$.

In order to demonstrate the stated inequality, first note that the $A$-module $D({_{A^{\textsf{o}}}A})$ is injective.  This gives the first equality below, the others are clear:  \begin{align*} \ed(D({_{A^{\textsf{o}}}A}) \oplus A) = \sup_i\{\Ext^i(D({_{A^{\textsf{o}}}A},A) \neq 0 \} \leq \pd(D({_{A^{\textsf{o}}}A})) = \id_{A^{\textsf{o}}}(A) \end{align*} 

For the last claim, we now have that if $\id_A(A)$ is finite and $A$ satisfies (GARC), then $A$ is Gorenstein.  In \cite{Zaks} it is shown that this implies $ \id_A(A) = \id_{A^{\textsf{o}}}(A)$. \end{proof}

\begin{rmk}  \label{rmk:op-ing}It is not known if satisfying the condition (GARC) is also a left-right symmetric property.  However, if $A$ is an Artin algebra then $\fed(A) = \fed(A^{\textsf{o}})$. Indeed, if $M$ is any $A$-module, we have \[ \Ext^i(_AM,_AM) \cong {\rm Ext}_{A^{\textsf{o}}}^i(D(_AM),D(_AM)) \]  It follows from this that $\ed_A(_AM) = \ed_{A^{\textsf{o}}}(D(_AM))$ for each $A$-module $M$.  This shows that $\fed(A) \leq \fed(A^{\textsf{o}})$ and an analogous argument will demonstrate the opposite inequality. 
\end{rmk}
\end{section}

\begin{section}{Applications for Gorenstein Rings}
In this section we restrict our attention to rings of finite injective dimension, and we will give several conditions that are equivalent to $\fed(A) < \infty$ for such a ring $A$.  In the next theorem we show that the following subcategory of $\Amod$ detects the finiteness of the finitistic extension degree:  \begin{align*} \CM = \{ M \ \in \Amod \ | \ \Ext^i(M,A) = 0 \ {\rm for} \ i > 0 \} \end{align*}  When $A$ is (commutative) Gorenstein, $\CM$ is the subcategory of (maximal) Cohen-Macaulay $A$-modules.  Note that $\CM$ is closed under taking syzygies and direct sums of modules.

\begin{thm}  \label{thm:equivalent} If $A$ is a Noetherian ring and $\id(A)$ is finite, then the following conditions are equivalent: 
\newcounter{count1}
\begin{list}{(\arabic{count1})}{\usecounter{count1} \leftmargin=.5em}
\item $\fed(A)$ is finite.
\item $\ed(M)$ is finite if and only if $\pd(M)$ is finite. 
\item If $M \in \CM$ then $\ed(M)< \infty$ if and only if $M$ is projective.
\item $\sup \{ \ed(M) \ | \ M \in \CM \ {\rm and} \ \ed(M) < \infty \} =0$.
\item $\fed(A) \leq \id(A)$.
\end{list}
\end{thm}

\begin{proof} We set $d = \id(A)$.  Then we have $\Ext^i(M,A) = 0$ for all $i > d$ and therefore $\ed(M) < \infty$ if and only if $\ed(M \oplus A)< \infty$. That (1) implies (2) now follows from Theorem \ref{thm:GARC}.  That (2) implies (3) follows from the definition of $\CM$ and (\ref{eq:pd.ext}).  Then (4) follows immediately from (3).  

To show that (4) implies (5), we assume that $M$ is an $A$-module and $\ed(M)$ is finite.  Then $\Omega^n(M) \in \CM$ for some $0 \leq n\leq d$ and Remark \ref{rmk:dimension.shift} gives that $\ed(\Omega^n(M)) < \infty$ and $\ed(M) \leq \ed(\Omega^n(M)) +d$.  The claim in (5) now follows from that in (4).
The final implication, from (5) to (1), is clear. \end{proof}
 
We point out separately the following converse to Corollary \ref{cor:GARC} for rings of finite injective dimension, which is simply a restatement of the equivalence of (1) and (2) in the Theorem.

\begin{cor} \label{cor:fed.garc}If $A$ is a Noetherian ring and $\id(A)$ is finite, then $\fed(A)$ is finite if and $A$ satisfies the Generalized Auslander-Reiten Condition. \end{cor}
   
In the case of commutative Gorenstein rings we can improve on the statement (3) from the theorem.  In \cite{varieties} it is shown that for a local complete intersection ring $A$ one has an equality $\pd(M) = \ed(M)$ for each $A$-module $M$. The next corollary shows that that this equality extends to all modules over any commutative Gorenstein ring of finite finitistic extension degree.  Before giving a proof of this, we record the following observation:

\begin{rmk}  \label{rmk:localization} If $A$ is a commutative Noetherian ring, then  \[ \pd_A(M)  =  \sup\{ \pd_{A_{\frak{m}}}(M_{\frak{m}}) \ | \ \frak{m} \ {\rm is \ a \ maximal \ ideal \ of } \ A \}.  \] Also, $\Ext^i(M,M) = 0$ if and only if ${\rm Ext}^i_{A_\frak{m}}(M_{\frak{m}}, M_{\frak{m}})=0$ for every maximal ideal $\frak{m}$.  Thus, we have an equality \[ \ed_A(M)  =  \sup\{ \ed_{A_{\frak{m}}}(M_{\frak{m}}) \ | \ \frak{m} \ {\rm is \ a \ maximal \ ideal \ of } \ A \}.  \]
From this, it also follows that \[ \fed(A) = \sup\{ \fed(A_{\frak{m}}) \ | \ \frak{m} \ {\rm is \ a \ maximal \ ideal \ of } \ A \}.  \]  \end{rmk}

\begin{cor} \label{cor:ed=pd} Assume that $A$ is a Noetherian ring and both $\id(A)$ and $\fed(A)$ finite.  If $A$ is either commutative or local, then $\ed(M) = \pd(M)$ for each $A$-module $M$.
\end{cor}

\begin{proof} When both $\id(A)$ and $\fed(A)$ are finite, Theorem 3.1 gives that $\ed(M)= \infty$ when $\pd(M)= \infty$.   If $A$ is local and $\pd(M)$ is finite, then Lemma \ref{lem:inequality} gives that $\pd(M) = \ed(M)$ when $\pd(M)$ is finite.  For commutative rings, use Remark \ref{rmk:localization} to reduce to the local case.
\end{proof}
 
In the next two corollaries, we show that the bound for $\fed(A)$ given in Theorem \ref{thm:equivalent}(5) is strict for commutative rings and Artin algebras having finite injective dimension and finite finitistic extension degree.

\begin{cor} \label{cor:fed.gor} Assume that $A$ is a commutative Noetherian ring.  If $\id_A(A)$ is finite, then $\fed(A) =  \id(A)$ or $\fed(A)=\infty$. \end{cor}

\begin{proof}  We first assume that $A$ is a local ring.  Since $\id(A)$ is finite, there exists a system of parameters $\x$ for $A$ and $\pd(A/(\x)) = \id(A)$ so that $\ed(A/(\x)) = \id(A)$ by Lemma \ref{lem:inequality}.  This give $\fed(A) \geq \id(A)$ and when $\fed(A)$ is finite Theorem \ref{thm:equivalent} provides the opposite inequality. 

If $A$ not a local ring one may use the equality  \begin{align*} \id_A(A) = \sup\{\id_{A_{\frak{m}}} \ | \ \frak{m} \ {\rm is \ a \ maximal \ ideal \ of \ } A\} \end{align*} and Remark \ref{rmk:localization} to again reduce to the local case.
\end{proof}

\begin{cor}\label{cor:fed.gor2} Assume that $A$ is an Artin algebra.  If $\id_A(A)$ is finite, then $\fed(A) =  \id(A)$ or $\fed(A)=\infty$. \end{cor}
\begin{proof} We assume that $A$ is an Artin algebra and that both $\id_A(A)$ and $\fed(A)$ are finite.  Then Corollary \ref{cor:fed.garc} gives that $A$ satisfies (GARC) so Proposition \ref{prop:symmetry} provides the equality $ \ed(D({_{A^{\textsf{o}}}A}) \oplus A) = \id_A(A)$.  This gives $\fed(A) \geq \id(A)$ and then Theorem \ref{thm:equivalent} provides the opposite inequality.  \end{proof}

The authors know of no commutative Gorenstein rings which have infinite finitistic extension degree.     We now describe an example, due to R.\ Schulz, of a non-commutative self-injective ring Artin algebra $A$ with $\fed(A) = \infty$.

\Ex \label{ex:fed.fails}Let $k$ be a field and $0\neq q \in k$ have infinite multiplicative order.  Set $ A = k\left<x,y\right> / (x^2,y^2,xy-qyx) $.  In \cite{Schulz} Schulz has shown that the $A$-module $M = A/(x+y)$ has $\ed(M) =1$.  Thus, $\fed(A) \geq 1$, and since $A$ is self-injective (see 3.1 in \cite{bergh}), Corollary \ref{cor:fed.gor2} gives $\fed(A) = \infty$.  This also shows that $A$ does not satisfy the Generalized Auslander-Reiten Condition.  \\

\begin{rmk} \label{rmk:arc.garc} The conditions (ARC) and (GARC) are not equivalent for Artin algebras.  We have seen that the ring $A$ in Example \ref{ex:fed.fails}  does not satisfy (GARC).  Since this is a local self-injective Artin algebra with maximal ideal $\frak{m}$ and $\frak{m}^3 =0$ it follows from Theorem 3.4 in \cite{Hoshino} that the Auslander-Reiten Condition holds for $A$. It remains unknown to the author if these conditions are equivalent for commutative (Gorenstein) rings, see also Remark 2.4 of \cite{AC}.  \end{rmk}
 
\begin{rmk} \label{rmk:fed.ee} We have observed that all (UAC) rings have finite finitistic extension degree but these conditions are not equivalent, cf.\ Remark \ref{rmk:AC.FED}.  We have shown that and many homological properties that hold for a ring $A$ satisfying (UAC) hold under the weaker assumtion that the $\fed(A) < \infty$.  However, not all properties of (UAC) rings follow from this weaker hypothesis.  Here we include an example of one such property.  

In \cite{HJ} it is shown that all commutative Gorenstein rings satisfying (AC) exhibit a symmetry property in the vanishing of Ext.  This symmetry property is called (ee) in \cite{ee}, where an example of a commutative, self-injective ring $A$ which does not satisfy property (ee) is given.  A slight modification of the proof of Theorem 3.3 in \cite{self-test} will show that the finitistic extension degree of this ring is zero (complete details are given in Theorem 4.21 of \cite{my.thesis}).  That is, $\fed(A) < \infty$ but $A$ does not satisfy the property (ee).  \end{rmk}
 
\end{section}

\begin{section}{Change of rings: The commutative case}

Here we examine how the finitistic extension degree behaves under adjoining variables, quotienting by a regular sequence and completion for commutative rings.  In this section we will always assume that the ring $A$ is commutative.  Similar results for rings satisfying (the Uniform) Auslander Condition have appeared in \cite{AC2}, \cite{HJ} and \cite{Ext-index}. 
 
\begin{prop} \label{regseq} Assume that $A$ is commutative Noetherian ring and $\x = a_1,...,a_n$ is an $A$-regular sequence. Then $\fed(A) \geq \fed(A/( \x)) +n$. \end{prop}

\begin{proof}  We show the case $n =1$, i.e. the regular sequence consists of a single element $a$.  A standard induction argument then gives the general case.  We set $\overline{A} = A/(a)$. It is enough to show that for any $\overline{A}$-module $M$ of finite extension degree there is an equality $\ed_A(M) = \ed_{\overline{A}}(M) +1$.   Given an $\overline{A}$-module $M$ there is a change of rings long exact sequence (see 11.65 in \cite{Rotman}): \[ \rightarrow {\rm Ext}^{i+1}_{\overline{A}}(M,M) \rightarrow \Ext^{i+1}(M,M) \rightarrow {\rm Ext}^{i}_{\overline{A}}(M,M) \rightarrow {\rm Ext}^{i+2}_{\overline{A}}(M,M) \rightarrow \] 

We set $m = \ed_{\overline{A}}(M)$.  If $x$ is finite, then the above sequence gives that   $ \Ext^{m+1}(M,M) \cong {\rm Ext}_{\overline{A}}^m(M,M) \neq 0 $ and $\Ext^i(M,M) = 0$ for all $i> m+1$.  This shows $\ed_A(M) = \ed_{\overline{A}}(M)+1$, as needed.  \end{proof}

When $A$ is a commutative local Gorenstein ring, we show that equality holds in the previous proposition.   

\begin{prop} \label{prop:regseqiff} If $A$ is a commutative local Gorenstein ring,  and $\x = a_1,...,a_n$ is an $A$-regular sequence, then $\fed(A) = \fed(A/(\x))+n$.  \end{prop}

\begin{proof}   If $\fed(A/(\x)) = \infty$, then Proposition \ref{regseq} gives that $\fed(A) = \infty$.  We may thus assume that $\fed(A)$ is finite. We first show that $\fed(A)$ is also finite.
Next, we assume that $\fed(A/(\x))$ is finite. We show that $\fed(A)$ is also finite, and then the desired equality follows from the previous paragraph.  By induction on the length of the sequence, it suffices to show the result when our sequence $\x$ is a single nonzero divisor $a$.  We set $\overline{X} = X/aX$.  Suppose that $M \in \CM$ and $M$ has finite extension degree.  We show that $M$ is free and then $\fed(A)$ is finite by Theorem \ref{thm:equivalent}.  For this, we have a short exact sequence: \begin{align*}  \xymatrix{ 0 \ar[r] & M \ar[r]^a & M \ar[r] & \overline{M} \ar[r] & 0 } \end{align*} and the resulting long exact sequence in $\Ext^*(-,M)$ gives $\Ext^i(\overline{M}, M) = 0$ for all $i > \ed(M)$.  Now 3.1.16 in \cite{BH}  gives $\Ext^{i+1}(\overline{M},M) \cong {\rm Ext}_{\overline{A}}^i(\overline{M},\overline{M}) = 0 $ for $i \geq \ed_A(M)$.  In particular, the extension degree of the maximal Cohen-Macaulay $\overline{A}$-module $\overline{M}$ is finite.  Theorem \ref{thm:equivalent} then gives that $\overline{M}$ is free over $\overline{A}$.  Lemma 1.3.5 in \cite{BH} then gives that $M$ is a free $A$-module, as claimed.

Corollary \ref{cor:fed.gor}  now gives the first and third equalities below:   \[ \fed(A) = \id_A(A) = \id_{A/(\x)}(A/(\x))+n = \fed(A/(\x))+n, \] the second equility is provided by Corollary 3.1.15 of \cite{BH}. \end{proof}

\begin{thm} \label{ringchg}  Assume that $A$ is a commutative local Gorenstein ring with maximal ideal $\frak{m}$ and that $X$ is an indeterminant.  If any of the following rings have finite finitistic extension degree, then they all must: \[A, \  \widehat{A}, \  A[[X]], \  A[X]_{(\frak{m},X)} \] \end{thm}

\begin{proof} Since $\widehat{A}$ is a faithfully flat $A$-module, we have that \[ 0=\Ext^i(M,M) \ {\rm if \ and \ only \ if} \ 0 = \Ext^i(M,M)\otimes_A \widehat{A} \cong {\rm Ext}_{\widehat{A}}^i(\widehat{M},\widehat{M}) \]  Thus $\fed(A)$ is finite when  $\fed(\widehat{A})$ is.  To see the converse, assume that $\fed(A)$ is finite and take a maximal $A$-sequence $\x$ and $A/(\x) \cong \widehat{A}/(\x) \widehat{A}$.  Then completing $\x$ gives rise to a maximal $\widehat{A}$-sequence.  From Proposition \ref{prop:regseqiff} we obtain $\fed(A/(\x)) < \infty$ so $\widehat{A}/(\x) \widehat{A})$ is also finite.  Applying Proposition \ref{prop:regseqiff} again gives $\fed(\widehat{A}) < \infty$.  

Next, observe that $X$ is a non-zerodivisor on $A[[X]]$ and $A[[X]]/X \cong A$, so that $\fed(A)$ and $\fed(A[[X]])$ are finite simultaneously by Proposition \ref{prop:regseqiff}.  Lastly, note that $\widehat{A[X]_{(\frak{m},X)}} \cong \widehat{A}[[X]]$ so that $\fed(A[X]_{(\frak{m},X)})$ is finite if and only if $\fed(\widehat{A}[[X]])$ is by the above.
\end{proof}
 
\begin{rmk} \label{CompleteIntersections} Recall that in \cite{varieties} it is shown that the equality $\pd(M) = \ed(M)$ holds for each $A$-module $M$ when $A$ is a local complete intersection ring.  We close this section with a direct proof of this fact that does not require the use of support varieties.  In view of Corollary \ref{cor:ed=pd}, it suffices to show that these rings have finite finitistic extension degree.

Indeed, if $B$ is a local complete intersection ring, then $\widehat{B} \cong A/(\x)$ for some regular local ring $A$ and an $A$-regular sequence $\x$.  Since $A$ has finite global dimension, it is clear that $\fed(A)< \infty$.  As we have $\widehat{B} \cong A/(\x)$, it now follows from Proposition \ref{regseq} that $\fed(\widehat{B})< \infty$.   Now Theorem \ref{ringchg} gives that $\fed(B)$ is finite because $\fed(\widehat{B})$ is.  

\end{rmk}

\end{section}

\section*{Acknowledgements}

The author would like to express his gratitude to his thesis advisor, Claudia Miller, for her guidance and support.  We also thank Luchezar Avramov, Lars Christensen, Srikanth Iyengar and Liana \c{S}ega for valuable feedback while this work was in progress.


\begin{thebibliography}{}
\bibitem{Aus1} Maurice Auslander,  {\it Selected works of Maurice Auslander. Part 1.} American Mathematical Society, Providence, RI (1999). Edited and with a foreword by Idun Reiten, Sverre O. Smal\o, and \O yvind Solberg.
\bibitem{AuR} Maurice Auslander and Idun Reiten, {\it Applications of contravariantly finite subcategories},  Adv. Math. {\bf 86} (1991), no. 1, 111-152. 
\bibitem{AuR2} Maurice Auslander and Idun Reiten, {\it On a generalized version of the Nakayama conjecture,} Proc. Amer. Math. Soc. {\bf 52},  (1975)  69-74.
\bibitem{ARS} Maurice Auslander, Idun Reiten and Sverre O. Smal\o,  {\it Representation theory of Artin algebras} Cambridge Studies in Advanced Mathematics, 36. Cambridge University Press, Cambridge, 1995.
\bibitem{varieties} Luchezar L. Avramov and Ragnar-Olaf Buchweitz, {\it Support varieties and cohomology over complete intersections}, Invent. Math. {\bf 142} (2000), no. 2, 285-318. 
\bibitem{bergh} Petter A. Bergh {\it Ext-symmetry over quantum complete intersections}, Arch. Math. (Basel) {\bf 92} (2009), no. 6, 566-573.
\bibitem{Bass} Hyman Bass, {\it Finitistic dimension and a homological generalization of semi-primary rings}. Trans. Amer. Math. Soc. {\bf 95} 1960 466-488.
\bibitem{BH} Winfried Bruns and J\"{u}rgen Herzog, {\it Cohen-Macaulay rings}, Cambridge Studies in Advanced Mathematics, 39. Cambridge University Press, Cambridge (1993).
\bibitem{CE} Henri Cartan and Samuel Eilenberg, {\it Homological algebra}. Princeton University Press, Princeton, N.J., 1956.
\bibitem{AC}  Lars W. Christensen and Henrik Holm, {\it Algebras that satisfy Auslander's condition on vanishing of cohomology}, Math. Z. {\bf 265}  (2010), no.1, 21-40.
\bibitem{AC2}  \underline{ \hspace{1cm} },  {\it Vanishing of cohomology over Cohen-Macaulay rings}, arXiv:1006.1006v1 
\bibitem{my.thesis} Kosmas Diveris, {\it On modules with eventually vanishing self-extensions}, Thesis, (2012).
\bibitem{D(GARC)} Kosmas Diveris and Marju Purin, {\it The Generalized AuslanderÐReiten Condition for the bounded derived category} Arch. Math. (Basel) {\bf 98}, (2012), no. 6, Page 507-511.
\bibitem{Hoshino} Mitsuo Hoshino, {\it Modules without self-extensions and Nakayama's conjecture}, Arch. Math. (Basel) {\bf 43} (1984), no. 6, 493-500.
\bibitem{sheerbrooke} Dieter Happel, {\it Homological conjectures in representation theory of finite-dimensional algebras} Sherbrook Lecture Notes Series (1991).
\bibitem{HJ} Craig Huneke and David A. Jorgensen,  Symmetry in the vanishing of Ext over Gorenstein rings. Math. Scand. {\bf 93} (2003), no. 2, 161-184.
\bibitem{aus} David A. Jorgensen and Liana M. \c{S}ega, {\it Nonvanishing cohomology and classes of Gorenstein rings}, Adv. Math. {\bf 188} (2004), 470-490.
\bibitem{ee} David A. Jorgensen and Liana M. \c{S}ega, {\it Asymmetric complete resolutions and vanishing of Ext over Gorenstein rings}, Int. Math. Res. Not., no. 56, (2005),  3459-3477.
\bibitem{Mori}  Izuru Mori, {\it Symmetry in the vanishing of Ext over stably symmetric algebras}, J. Algebra {\bf 310} (2007), no. 2, 708-729.
\bibitem{Ext-index}  Saeed Nasseh and Yuji Yoshino, {\it  On Ext-indices of ring extensions},  J. Pure Appl. Algebra {\bf 213} (2009), no. 7, 1216-1223.
\bibitem{Rotman} Joseph J. Rotman,  {\it An Introduction to Homological Algebra}, Academic Press, New York, 1979. 
\bibitem{Schulz} Ranier Schulz, {\it A nonprojective module without self-extensions}, Arch. Math. (Basel) {\bf 62} (1994), no. 6, 497-500.
\bibitem{self-test} Liana M. \c{S}ega, {\it Self-tests for freeness over commutative Artinian rings}, J. Pure Appl. Algebra {\bf 215} (2011), no. 6, 1263-1269.
\bibitem{Weibel} Charles A. Weibel,  {\it An introduction to homological algebra}, Cambridge Studies in Advanced Mathematics, 38. Cambridge University Press, Cambridge, 1994.
\bibitem{Zaks} Abraham Zaks, {\it Injective dimension of semi-primary rings}, J. Algebra {\bf 13} 1969 73-86.

\end{thebibliography}
\end{document}